\documentclass[a4,12pt]{article}
\usepackage{amsmath}
\usepackage{amsthm}
\usepackage{dsfont}
\usepackage{latexsym, amssymb, amsthm}
\usepackage{mathrsfs}
\usepackage{amsxtra}
\usepackage{amstext}
\usepackage{amscd}
\usepackage{mathrsfs}
\usepackage{amsxtra}
\usepackage{amstext}
\usepackage{amscd}
\usepackage{bbold}
\usepackage{euscript}

\newtheorem{theorem}{Theorem}[section]
\newtheorem{corollary}[theorem]{Corollary}
\newtheorem{lemma}[theorem]{Lemma}
\newtheorem{proposition}[theorem]{Proposition}

\newtheorem{definition}[theorem]{Definition}

\usepackage{dsfont}
\numberwithin{equation}{section}
\begin{document}

\title{Comments on the Chernoff $\sqrt{n}$-Lemma}

\author{ Valentin A. Zagrebnov\\
Aix-Marseille Universit\'{e}, CNRS, Centrale Marseille, I2M UMR 7373 \\
13453 Marseille, France  \\
{valentin.zagrebnov@univ-amu.fr}
\thanks{The author is grateful to Mathematical Department of the University of Auckland (New Zealand)
for hospitality during the writing this paper. This visit was supported by the EU Marie Curie
IRSES program, project `AOS', No. 318910.}
}

\maketitle

\textit{Dedicated to Pavel Exner in occasion of his 70th Birthday}

\begin{abstract}
\noindent The Chernoff $\sqrt{n}$-Lemma is revised. This concerns two aspects: an improvement of the
Chernoff estimate in the strong operator topology and an operator-norm estimate for quasi-sectorial contractions.
Applications to the Lie-Trotter product formula approximation for semigroups is presented.
\end{abstract}

\noindent Classification: primary 47D06; secondary 47D60.


\maketitle

\section{Introduction: $\sqrt{n}$-Lemma}

The Chernoff $\sqrt{n}$-Lemma is a key point in the theory of semigroup approximations proposed in \cite{Ch}.
For the reader convenience we recall this lemma below.

\begin{lemma}\label{lem:2.1.9}
Let $C$ be a contraction on a Banach space $\mathfrak{X}$. Then $\{e^{t (C - \mathds{1}}\}_{t\geq 0}$ is
a norm-continuousn contraction semigroup on $\mathfrak{X}$ and one has the estimate
\begin{equation}\label{eq:2.1.10}
\|(C^n - e^{n(C-\mathds{1})})x\| \leq \ \sqrt{n} \ \|(C-\mathds{1})x\| \ \ .
\end{equation}
for all $x\in\mathfrak{X} $ and $\ n \in \mathbb{N}$.
\end{lemma}
%
\begin{proof}
To prove the inequality (\ref{eq:2.1.10}) we use the representation
\begin{equation}\label{eq:2.1.11}
C^n-e^{n(C-\mathds{1})} = e^{-n}\sum_{m=0}^\infty {n^m\over m!}(C^n-C^m) \ .
\end{equation}
To proceed we insert
\begin{equation}\label{eq:2.1.12}
\|(C^n-C^m)x\|  \leq  \left\|(C^{|n-m|}  - \mathds{1})x \right\|\leq  |m-n| \|(C- \mathds{1})x \| \ ,
\end{equation}
into (\ref{eq:2.1.11}) to obtain by the Cauchy-Schwarz inequality  the estimate:
\begin{equation}\label{eq:2.1.13}
\begin{split}
&\|(C^n - e^{n(C-\mathds{1})})x \| \leq \|(C - \mathds{1})x \| \  e^{-n} \ \sum_{m=0}^\infty {n^m\over m!} |m-n| \leq\\
&  \{\sum_{m=0}^\infty  e^{-n} \  {n^m\over m!} |m-n|^2\} ^{1/2} \|(C - \mathds{1})x\| \ , \ x\in\mathfrak{X} \ ,
\end{split}
\end{equation}
Note that the sum in the right-hand side of (\ref{eq:2.1.13}) can be calculated explicitly. This gives the value $n$,
which yields (\ref{eq:2.1.10}).
\end{proof}

The aim of the present note is to revise the Chernoff $\sqrt{n}$-Lemma in two directions.
First, we improve the $\sqrt{n}$-estimate (\ref{eq:2.1.10}) for contractions. Then we apply this new estimate
to the proof of the Trotter product formula in the \textit{strong} operator topology (Section 2).

Second, we use the idea of Section 2 to lift these results in Section 3 to the \textit{operator-norm} estimates
for a special class of contractions: the \textit{quasi-sectorial} contractions.

\section{Revised $\sqrt{n}$-Lemma and Lie-Trotter product formula}
We start by a technical lemma. It is a \textit{revised} version of the
Chernoff $\sqrt{n}$-Lemma \ref{lem:2.1.9}. Our estimate (\ref{eq:1.8.51}) in $\sqrt[3]{n}$-Lemma \ref{lem:1.8.23} is
better than (\ref{eq:2.1.10}). The scheme of the proof will be useful later (Section 3), when we use it for proving the
convergence of Lie-Trotter product formula in the \textit{operator-norm} topology.
\begin{lemma}\label{lem:1.8.23}
Let $C$ be a contraction on a Banach space $\mathfrak{X}$. Then $\{e^{t (C - \mathds{1}}\}_{t\geq 0}$ is a norm-continuous
contraction semigroup on $\mathfrak{X}$ and one has the estimate
\begin{equation}\label{eq:1.8.51}
\|(C^n-e^{n(C-\mathds{1})})x\|\leq \left[{1\over n^{2 \delta}} + n^{\delta +
1/2}\right]\|(\mathds{1}-C)x\|, \ \ n\in \mathbb{N}.
\end{equation}
for all $x\in \mathfrak{X}$ and $\delta \in \mathbb{R}$.
\end{lemma}
\begin{proof}
Since the operator $C$ is bounded and $\|C\|\leq 1$, $(\mathds{1}-C)$ is a generator of the norm-continuous
semigroup, which is also a contraction:
\begin{equation*}
\|e^{\ t (C-\mathds{1})}\| \leq e^{-t} \|\sum_{m=0}^{\infty} \frac{t^m}{m!} C^m\| \leq 1 \ , \ \ \ t \geq 0 \ .
\end{equation*}
To estimate (\ref{eq:1.8.51}) we use the representation
\begin{equation}\label{eq:1.8.52}
C^n-e^{n(C-\mathds{1})} = e^{-n}\sum_{m=0}^\infty {n^m\over m!}(C^n-C^m) \ .
\end{equation}

Let $\epsilon_n:=n^{\delta +1/2}$, $n\in \mathbb{N}$. We split the sum (\ref{eq:1.8.52}) into two parts: the
\textit{central} part for $|m-n|\leq\epsilon_n$ and \textit{tails} for $|m-n|>\epsilon_n$.

To estimate the \textit{tails} we use the \textit{Tchebychev inequality}. Let $X_n$ be a \textit{Poisson random variable}
of the parameter $n$, i.e., the probability $\mathbb{P}\{X_{n} = m\} = n^m e^{-n}/m! $.
One obtains for the expectation ${\mathbb{E}}(X_{n})=n$ and for the variance $\mbox{Var}(X_{n})=n$. Then by the Tchebychev
inequality:
\begin{equation*}
 \ \ \mathbb{P}\{|X_{n} - {\mathbb{E}}(X_{n})|>\epsilon\} \leq {\mbox{Var}(X_{n})\over\epsilon^2} \ , \
 \ {\rm{for \ any}\ } \ \epsilon>0 \ .
\end{equation*}

Now to estimate (\ref{eq:1.8.52}) we note that
\begin{eqnarray*}
\|(C^n-C^m)x\| & = & \|C^{n-k}(C^{k}-C^{m-n+k})x\|\\
& \leq & |m-n| \|C^{n-k}(\mathds{1}-C) x\| \ , \ \ \ k= 0,1,\ldots,n \ .
\end{eqnarray*}
Put in this inequality $k = [\epsilon_n]$, here $[\cdot]$ denotes the integer part. Then by $\|C\|\leq 1$ and
by the Tchebychev inequality we obtain the  estimate for \textit{tails}:
\begin{eqnarray}\label{eq:1.8.53}
&& \hspace{2cm} e^{-n}\sum_{|m-n|>\epsilon_n} {n^m\over m!}\|(C^n-C^m)x\| \leq  \\
&&\|(\mathds{1}-C) x\| \ e^{-n}\sum_{|m-n|>\epsilon_n} {n^m\over m!} |m-n|
\leq {n\over\epsilon_n^2} \|(\mathds{1}-C) x\| = {1 \over n^{2 \delta}} \|(\mathds{1}-C)x\| \ .\nonumber
\end{eqnarray}

To estimate the \textit{central} part of the sum (\ref{eq:1.8.52}), when $|m-n|\leq\epsilon_n$, note that:
\begin{eqnarray}\label{eq:1.8.54}
\|(C^n-C^m)x\| &\leq & |m-n| \|C^{n-[\epsilon_n]}(\mathds{1}-C) x\| \\
& \leq & \epsilon_n \ \|(\mathds{1}-C)x\| \ . \nonumber
\end{eqnarray}
Then we obtain:
\begin{equation*}
e^{-n}\sum_{|m-n|\leq\epsilon_n} {n^m\over m!}\|(C^n-C^m)x\| \leq n^{\delta +1/2} \ \|(\mathds{1}-C)x\| \ ,
\end{equation*}
for $n\in {\mathbb{N}}$. This last estimate together with (\ref{eq:1.8.53}) yield (\ref{eq:1.8.51}).
\end{proof}
Note that for $\delta= 0$ the estimate (\ref{eq:1.8.51}) gives for large $n$ the same asymptotic as the Chernoff
$\sqrt{n}$-Lemma, whereas for optimal value $\delta= (-1/6)$ the asymptotic $2 \sqrt[3]{n}$ is better than
(\ref{eq:2.1.10}). We call this result the $\sqrt[3]{n}$-Lemma.
\begin{theorem}\label{lem:1.8.24}
Let $\Phi: t\mapsto\Phi(t)$ be a function from $\mathbb{R}^+$ to contractions on $\mathfrak{X}$ such that
$\Phi(0) = \mathds{1}$. Let $\{U_{A}(t)\}_{t\geq 0}$ be a contraction semigroup, and let $D \subset {\rm{dom}}(A)$ be
a core of the generator $A$. If the function $\Phi(t)$ has a strong right-derivative $\Phi'(+0)$ at $t=0$ and
\begin{equation*}
\Phi'(+0)u := \lim_{t\rightarrow +0} \frac{1}{t} (\Phi(t)- \mathds{1}) u = -\, A u \ ,
\end{equation*}
for all $u\in D$, then
\begin{equation}\label{eq:1.8.55}
\lim_{n \rightarrow \infty} [\Phi(t/n)]^n \, x = U_{A}(t)\, x  \ ,
\end{equation}
for all $t\in \mathbb{R}^+$ and $x\in \mathfrak{X}$
\end{theorem}
\begin{proof}
Consider the bounded approximation $A_n$ of generator $A$:
\begin{equation}\label{eq:1.8.56}
A_n(s) := \frac{\mathds{1} - \Phi(s/n)}{s/n} \ .
\end{equation}
This operator is \textit{accretive}:  $(A_n(s) + \zeta \mathds{1})^{-1} \in \mathcal{L}(\mathfrak{X})$ and
$\|(A_n(s) + \zeta \mathds{1})^{-1}\| \leq ({\rm{Re}} (\zeta))^{-1}$ for ${\rm{Re}}(\zeta) > 0$, and
\begin{equation}\label{eq:1.8.57}
\lim_{n \rightarrow \infty} A_n(s) \, u = A \, u \ ,
\end{equation}
for all $u \in D $ and for bounded $s$. Therefore, by virtue of the Trotter-Neveu-Kato generalised strong convergence
theorem one gets:
\begin{equation}\label{eq:1.8.58}
\lim_{n \rightarrow \infty} e^{- t \, A_n(s)} \, x = U_{A}(t)\, x \ ,
\end{equation}
i.e., the strong and the uniform in $t$ and $s$ convergence (\ref{eq:1.8.58}) of the approximants
$\{e^{- t \, A_n(s)}\}_{n\geq 1}$ for $s \in {(0, s_{0}]}$.
By Lemma \ref{lem:1.8.23} for contraction $C := \Phi(t/n)$ and for $ A_n(s)|_{s=t}$ we obtain
\begin{eqnarray}\label{eq:1.8.59}
&& \|[\Phi(t/n)]^n \, x - e^{- t \, A_n(t)} \, x\| = \|([\Phi(t/n)]^n - e^{{n}(\Phi(t/n)- \mathds{1})}) \ x \| \leq \\
&&\hspace{3.5cm}\leq{2\over n^{2 \delta}}\|x\| + n^{\delta + 1/2}\|(\mathds{1}- \Phi(t/n)) \, x \| \ . \nonumber
\end{eqnarray}
Since for any $u \in D $ and uniformly on $[0,t_0]$ one gets
\begin{equation}\label{eq:1.8.60}
\lim_{n \rightarrow \infty} n^{\delta + 1/2}\|(\mathds{1}- \Phi(t/n)) \, u\| =
\lim_{n \rightarrow \infty} t \, n^{\delta - 1/2}\|A_n(t) \, u\| = 0 \ ,
\end{equation}
for $\delta < 1/2\ $, equations (\ref{eq:1.8.59}) and (\ref{eq:1.8.60}) imply
\begin{equation}\label{eq:1.8.61}
\lim_{n \rightarrow \infty}\|[\Phi(t/n)]^n \, u - e^{- t \, A_n(t)} \, u\| = 0 , \ \  u \in D \ .
\end{equation}
Then (\ref{eq:1.8.58}) and (\ref{eq:1.8.61}) together with estimate $\|[\Phi(t/n)]^n  - e^{- t \, A_n(t)}\| \leq 2 $
yield uniformly in $t\in [0, t_0]$:
\begin{equation*}
\lim_{n \rightarrow \infty} [\Phi(t/n)]^n \, x = U_{A}(t)\, x  \ ,
\end{equation*}
which by density of $D$ is extended to all $x\in \mathfrak{X}$, cf (\ref{eq:1.8.55}).
\end{proof}

{{We call (\ref{eq:1.8.55}) the (strong) \textit{Chernoff approximation formula} for the
semigroup $\{U_{A}(t)\}_{t\geq 0}$.}}

\begin{proposition}\label{prop:1.8.25}{\rm{\cite{Ch}}}{\rm{(Lie-Trotter product formula)}}
Let $A$, $B$ and $C$ be generators of contraction semigroups on $\mathfrak{X}$. Suppose that algebraic sum
\begin{equation}\label{eq:1.8.62}
C u = A u + Bu \ ,
\end{equation}
is valid for all vectors $u$ in a core $D\subset {\rm{dom}} \, C$. Then the semigroup  $\{U_{C}(t)\}_{t\geq 0}$ can be
approximated on $\mathfrak{X}$ in the strong operator topology (\ref{eq:1.8.59}) by the Lie-Trotter product
formula:
\begin{equation}\label{eq:1.8.63}
e^{- t C} \, x = \lim_{n\rightarrow \infty} (e^{- t A/n} e^{- t  B/n} )^n \, x  \ , \  \  \ x\in\mathfrak{X} \ ,
\end{equation}
for all $t\in \mathbb{R}^{+}$ and for $C := \overline{(A + B)}$, which  is the closure of the algebraic sum
(\ref{eq:1.8.62}).
\end{proposition}
\begin{proof}
Let us define the contraction $\mathbb{R}^{+}\ni t \mapsto \Phi(t)$, $\Phi(0)= \mathds{1}$, by
\begin{equation}\label{eq:1.8.64}
\Phi(t) := e^{- t A} e^{- t B} \ .
\end{equation}
Note that if $u \in D$, then derivative
\begin{equation}\label{eq:1.8.65}
\Phi'(+0)u = \lim_{t\rightarrow +0} \frac{1}{t} ( \Phi(t)-\mathds{1})\ u = -(A + B)\ u \ .
\end{equation}

Now we are in position to apply Theorem \ref{lem:1.8.24}. This yields (\ref{eq:1.8.63}) for
$C := \overline{(A + B)}$.
\end{proof}

\begin{corollary}\label{cor:1.8.26}
Extension of the strong convergent Lie-Trotter product formula of Proposition \ref{prop:1.8.25} to quasi-bounded
and holomorphic semigroups goes through verbatim.
\end{corollary}

\section{Quasi-sectorial contractions: $(\sqrt[3]{n})^{-1}$-Theorem}
\label{subsec:2.1.2}
\begin{definition}\label{def:2.1.2} \cite{CaZ}
A {contraction} $C$ on the Hilbert space $\mathfrak{H}$ is called
\textit{quasi-sectorial} with  semi-angle $\alpha\in [0, \pi/2)$
with respect to the vertex at $z=1$, if its numerical range $W(C)\subseteq D_{\alpha}$.
Here
\begin{equation}\label{eq:2.1.1}
 D_{\alpha}:=\{z\in {\mathbb{C}}: |z|\leq \sin \alpha\} \cup
\{z\in {\mathbb{C}}: |\arg (1-z)|\leq \alpha \ {\rm{and}}\ |z-1|\leq
\cos \alpha \} .
\end{equation}
\end{definition}

We comment that $D_{\alpha = \pi/2} = {\mathbb{D}}$ (unit disc) and recall that a general contraction $C$ verifies the
weaker condition: $W(C)\subseteq {\mathbb{D}}$.

Note that if operator $C$ is a quasi-sectorial contraction, then $\mathds{1}- C$ is an $m$-sectorial operator with
vertex $z=0$ and semi-angle $\alpha$. Then for $C$ the limits: $\alpha=0$ and $\alpha = \pi/2$,
correspond respectively to self-adjoint and to standard contractions whereas for $\mathds{1}- C$
they give a non-negative self-adjoint and an $m$-accretive (bounded) operators.

The resolvent of an $m$-sectorial operator $A$, with {semi-angle} $\alpha\in [0, \alpha_0]$, $\alpha_0 < \pi/2$,
and vertex at $z=0$, gives an example of the quasi-sectorial contraction.

\begin{proposition}\label{prop:2.1.10}{\rm{\cite{CaZ, Zag}}}
If $C$ is a quasi-sectorial contraction on a Hilbert space $\mathfrak{H}$ with semi-angle
$0\leq\alpha < \pi/2$, then
\begin{equation}\label{eq:2.1.14}
\|C^n (\mathds{1}-C)\|\leq \frac{K}{n+1} \ , \ n\in{\mathbb{N}} \ .
\end{equation}
\end{proposition}

The property (\ref{eq:2.1.14}) implies that the quasi-sectorial contractions belong to the class of the so-called
{\textit{Ritt operators}} \cite{Ri}. This allows to go beyond the $\sqrt[3]{n}\ $-Lemma \ref{lem:1.8.23}
to the $(\sqrt[3]{n})^{-1}$-Theorem.

\begin{theorem}\label{th:6.2.2}{\rm{($(\sqrt[3]{n})^{-1}$-Theorem)}}
Let $C$ be a quasi-sectorial contraction on $\mathfrak{H}$ with
numerical range $W(C)\subseteq D_\alpha$, $0\leq \alpha <\pi/2$. Then
\begin{equation}\label{eq:6.2.5}
\left\|C^n - e^{n(C-\mathds{1})}\right\| \leq {M\over n^{1/3}} \ , \ \ n=1,2,3,\dots
\end{equation}
where $M=2K+2$ and $K$ is defined by (\ref{eq:2.1.14}).
\end{theorem}
\begin{proof}
Note that with help of inequality (\ref{eq:2.1.14}) we can improve the
estimate (\ref{eq:1.8.54}) in Lemma \ref{lem:1.8.23}:
\begin{eqnarray*}
\|C^n-C^m\| \leq  |m-n| \|C^{n-[\epsilon_n]}(\mathds{1}-C)\|
\leq \epsilon_n \ \frac{K}{n-[\epsilon_n]+1} \ ,
\end{eqnarray*}
for $\epsilon_n = n^{\delta + 1/2}$. Then for $\delta < 1/2$ there the above
inequality together with (\ref{eq:1.8.53}) give instead of (\ref{eq:1.8.51}) (or (\ref{eq:2.1.10})) the
operator-norm estimate
\begin{equation}\label{eq:2.1.15}
\left\|C^n - e^{n(C-\mathds{1})}\right\| \leq \frac{2}{n^{2\delta}} + \frac{2 K}{n^{1/2 -\delta}}
\ \  , \ \ n \in \mathbb{N} \ .
\end{equation}
Then the estimate ${M}/{n^{1/3}}$ of the Theorem  \ref{th:6.2.2} results from the optimal choice of the value:
$\delta = 1/6 $, in (\ref{eq:2.1.15}).
\end{proof}

Similar to $(\sqrt[3]{n})$-Lemma, the $(\sqrt[3]{n})^{-1}$-Theorem is the first step in developing the
\textit{operator-norm} approximation formula \`{a} la Chernoff. To this end one needs an operator-norm analogue of
Theorem \ref{lem:1.8.24}.
Since the last includes the Trotter-Neveu-Kato strong convergence theorem, we need the operator-norm
extension of this assertion to quasi-sectorial contractions.
\begin{proposition}\label{lem:6.3.1}{\rm{\cite{CaZ}}}
Let $\{X(s)\}_{s>0}$ be a family of m-sectorial operators in a Hilbert space $\mathfrak{H}$ with
$W(X(s))\subseteq S_\alpha$ for some $0< \alpha <\pi/2$ and for all $s>0$. Let $X_0$ be an
$m$-sectorial operator defined in a closed subspace ${\mathfrak{H}}_0 \subseteq {\mathfrak{H}}$, with
$W(X_0)\subseteq S_\alpha$. Then the two following assertions are equivalent:
\begin{eqnarray*}
(a) & & \lim_{s\rightarrow +0} \left\|(\zeta \mathds{1} +X(s))^{-1} -
(\zeta \mathds{1} +X_0)^{-1}P_0\right\| = 0 \ , \ \mbox{ for } \ \zeta\in S_{\pi-\alpha} \ ,  \\
(b) & & \lim_{s\rightarrow +0} \left\|e^{-tX(s)} - e^{-tX_0}P_0\right\| = 0 \ , \
\mbox{ for } \ t> 0 \ .
\end{eqnarray*}
Here $P_0$ denotes the orthogonal projection from $\mathfrak{H}$ onto ${\mathfrak{H}}_0$.
\end{proposition}

Now $(\sqrt[3]{n})^{-1}$-Theorem \ref{th:6.2.2} and Proposition \ref{lem:6.3.1} yield a desired generalisation of the
operator-norm approximation formula:
\begin{proposition}\label{prop:2.1.12}{\rm{\cite{CaZ}}}
Let $\{\Phi(s)\}_{s\geq 0}$ be a family of uniformly
quasi-sectorial contractions on a Hilbert space $\mathfrak{H}$, i.e. such that there exists
$0 \leq \alpha<\pi/2$ and  $W(\Phi(s)) \subseteq D_\alpha$, for all $s\geq 0$. Let
\begin{equation}\label{eq:2.1.16}
X(s):=(\mathds{1}-\Phi(s))/s  \ ,
\end{equation}
and let $X_0$ be a closed operator with non-empty resolvent set, defined in a closed subspace
${\mathfrak{H}}_0 \subseteq \mathfrak{H}$. Then the family $\{X(s)\}_{s>0}$ converges,
when $s\rightarrow +0$, in the uniform resolvent sense to the operator $X_0$  if and only if
\begin{equation}\label{eq:2.1.17}
\lim_{n\rightarrow \infty} \left\|\Phi(t/n)^n -e^{-tX_0}P_0\right\| = 0 \ , \ \ \ \mbox{for} \ t>0 \ .
\end{equation}
Here $P_0$ denotes the orthogonal projection onto the subspace ${\mathfrak{H}}_0$.
\end{proposition}

Let $A$ be an $m$-sectorial operator with semi-angle $0< \alpha<\pi/2$ and with vertex at $0$, which means that
numerical range $W(A)\subseteq S_\alpha = \{z\in \mathbb{C}: |\arg(z)| \leq \alpha\}$.
Then $\{\Phi(t) := (\mathds{1}+tA)^{-1}\}_{t\geq 0}$ is the family of quasi-sectorial contractions, i.e.
$W(\Phi(t))\subseteq D_\alpha$. Let $X(s):= (\mathds{1}-\Phi(s))/s$, $s>0$, and $X_0:=A$. Then $X(s)$ converges,
when $s\rightarrow +0$, to $X_0$ in the uniform resolvent sense with the asymptotic
\begin{equation*}
\|(\zeta \mathds{1} +X(s))^{-1} - (\zeta \mathds{1} +X_0)^{-1}\| = s\left\| {A\over\zeta \mathds{1} +A+
\zeta sA}\cdot{A\over\zeta \mathds{1} +A}\right\| = O(s),
\end{equation*}
for any $\zeta\in S_{\pi-\alpha}$, since we have the estimate:
\begin{equation*}
\left\| {A\over\zeta \mathds{1} +A+\zeta sA}\cdot{A\over\zeta \mathds{1} +A}\right\| \leq \left(1+
{|\zeta|\over \mbox{dist}\left({\zeta (1+s\zeta)^{-1}},-S_\alpha\right)}\right)
\left(1+{|\zeta|\over \mbox{dist} (\zeta,-S_\alpha)}\right) \ .
\end{equation*}

Therefore, the family $\{\Phi(t)\}_{t\geq 0}$ satisfies the conditions of Proposition \ref{prop:2.1.12}. This
implies the operator-norm approximation of the exponential function , i.e. the semigroup for $m$-sectorial
generator, by the powers of resolvent (\textit{the Euler approximation formula}):
\begin{corollary}\label{th:6.4.1}
If $A$ is an $m$-sectorial operator in a Hilbert space $\mathfrak{H}$, with semi-angle $\alpha\in (0,\pi/2)$ and
with vertex at $0$, then
\begin{equation}\label{eq:E}
\lim_{n\rightarrow\infty}\left\|(\mathds{1}+tA/n)^{-n} - e^{-tA}\right\| = 0\ , \ \ t\in S_{\pi/2-\alpha} \ .
\end{equation}
\end{corollary}
\section{Conclusion}

Summarising we note that for the {quasi-sectorial} contractions instead of \textit{divergent} Chernoff's estimate
(\ref{eq:2.1.10})  we find the estimate (\ref{eq:2.1.15}), which converges for $n \rightarrow \infty$ to zero in the
\textit{operator-norm} topology.
Note that the rate $O(1/n^{1/3})$ of this convergence is obtained with help of the
{Poisson representation} and the {Tchebychev inequality} in the spirit of the proof of Lemma \ref{lem:1.8.23},
and that it is not optimal.

The estimate ${M}/{n^{1/3}}$ in the $(\sqrt[3]{n})^{-1}$-Theorem \ref{th:6.2.2} can be improved by a more
refined lines of reasoning.

For example, by scrutinising our probabilistic arguments one can find a more
precise Tchebychev-type bound for the tail probabilities. This improves the estimate (\ref{eq:2.1.15}) to
the rate $O(\sqrt{\ln(n)/n})$, see \cite{Pa}.

On the other hand, a careful analysis of localisation the numerical range of quasi-sectorial contractions
\cite{Zag, ArZ}, allows to lift the estimate in Theorem \ref{th:6.2.2} and in Corollary \ref{th:6.4.1}
to the ultimate optimal rate $O(1/n)$.

Note that the optimal estimate $O(1/n)$ in (\ref{eq:2.1.15}) one can easily obtain with help of the
spectral representation for a particular case of the \textit{self-adjoint} quasi-sectorial contractions,
i.e. for $\alpha=0$. This also concerns the optimal $O(1/n)$ rate of convergence of the self-adjoint Euler
approximation formula (\ref{eq:E}).

\frenchspacing


\end{document}